\documentclass{amsart}

\usepackage{mathrsfs}
\usepackage{amssymb,latexsym}
\usepackage[all]{xy}
\usepackage{enumerate}

\usepackage[usenames]{color}

\definecolor{dgreen}{rgb}{0.1,0.7,0.1}
\definecolor{purple}{rgb}{0.49, 0.06, 0.51}

\allowdisplaybreaks

\numberwithin{equation}{section}

\newcommand{\Q}{\mathbb{Q}}
\def\N{\mathbb{N}}
\def\R{\mathbb{R}}
\def\Z{\mathbb{Z}}

\def\la{\langle}
\def\ra{\rangle}
\def\s{\sigma}

\newcommand{\vt}{\vartheta}
\newcommand{\ox}{\otimes}
\newcommand{\x}{\times}
\newcommand{\sm}{\setminus}
\newcommand{\Nil}{\mathrm{Nil}}
\newcommand{\wt}{\widetilde}
\newcommand{\id}{\mathrm{id}}
\newcommand{\vf}{\varphi}

\newcommand{\too}{\longrightarrow}

\DeclareMathOperator{\im}{Im}

\DeclareMathOperator{\sign}{sign}
\DeclareMathOperator{\sgn}{sgn}
\DeclareMathOperator{\Sym}{Sym}

\DeclareMathOperator{\st}{st}

\newcommand{\trs}[1]{\mathrm{Tr}^*_{#1}\,}

\renewcommand{\leq}{\leqslant}

\renewcommand{\le}{\leqslant}

\newcommand{\swap}{\widehat{\phantom{x}}}

\newcommand{\df}{\emph}
\newcommand{\ve}{\varepsilon}

\newcommand{\rd}{\mathrm{red}}
\newcommand{\tor}{\mathrm{t}}

\DeclareMathOperator{\dotcup}{\dot\cup}

\DeclareMathOperator*{\bigdotcup}{\dot\bigcup}

\newtheorem{thm}{Theorem}[section]
\newtheorem{prop}[thm]{Proposition}
\newtheorem{cor}[thm]{Corollary}
\newtheorem{lemma}[thm]{Lemma}
\theoremstyle{definition}
\newtheorem{defi}[thm]{Definition}

\newtheorem{remark}[thm]{Remark}

\newtheorem{examples}[thm]{Examples}

\begin{document}
\title{Stability index for algebras with involution}
\author{Vincent Astier}
\address{School of Mathematics and Statistics\\ University College Dublin\\ Belfield\\
Dublin~4\\ Ireland} 
\email{vincent.astier@ucd.ie, thomas.unger@ucd.ie}

\author{Thomas Unger}


\begin{abstract}
In earlier work we developed the theory of signatures 
of hermitian forms over algebras with involution with respect to orderings on the base field of the algebra
and obtained in particular that the total signature of a hermitian form is a continuous function from the space of orderings
of that field to $\Z$. 
In this note we give another presentation of signatures and also introduce and study  the stability index of algebras with involution.
\end{abstract}

\maketitle

\section{Introduction}

Let $F$ be a field of characteristic not two with space of orderings $X_F$ and Witt ring $W(F)$. Let 
$C(X_F,\Z)$ denote the ring of continuous functions from $X_F$ (equipped with the Harrison topology)  to $\Z$ (equipped with the discrete topology). 
The Sylvester signature $\sign_P: W(F)\to \Z$
at a given ordering $P\in X_F$
gives rise to
the well-known
exact sequence
\begin{equation}\label{classical}
\xymatrix{
0 \ar[r] &   W_\tor(F ) \ar[r] & W(F) \ar[r]^--{\sign} &  C(X_F,\mathbb{Z}) \ar[r] & S(F)\ar[r] & 0,
}
\end{equation}
where $S(F)$ denotes the stability group of $F$ and $W_\tor(F)$ the torsion part of $W(F)$. Exactness at $W(F)$ is Pfister's local-global principle. Furthermore,  
$W_\tor(F)$ and $S(F)$ are $2$-primary torsion groups, cf. \cite[\S 1]{Broecker-74}.

In \cite{A-U-Kneb} and \cite{A-U-prime} we developed the theory of signatures 
of hermitian forms over $F$-algebras with involution with respect to orderings on $F$
and obtained in particular that the total signature of a hermitian form is a continuous function from $X_F$ to $\Z$. 

In this paper we give another presentation of these signatures and use them to
introduce and study  the stability index of algebras with involution. 
We show that the resulting stability group is $2$-primary torsion and  obtain, using Pfister's local-global principle
for hermitian forms (\cite[Theorem~4.1]{LU1} and \cite[Theorem~5.1]{Sch1}),  an exact sequence  that extends \eqref{classical} to Witt groups of algebras with involution, cf. Theorem~\ref{thm:main}.

\section{Notation}\label{sec:notation}

We give an overview of the notation used in this paper and refer to the standard references 
\cite{Knus}, \cite{BOI}, \cite{Lam} and \cite{Sch} as well as \cite{A-U-Kneb} and \cite{A-U-prime}
for the details.

For a ring $A$ and an involution $\s$ on $A$, we denote
the set of symmetric elements of $A$ with respect to $\s$ by $\Sym(A,\s) =\{a\in A\mid \s(a)=a\}$.
For $\ve\in\{-1,1\}$ we write $W_\ve(A,\s)$ for the \df{Witt group} of Witt equivalence classes of $\ve$-hermitian forms $h:M\x M\to A$, defined on 
finitely generated projective right $A$-modules $M$. We drop the subscript $\ve$ when $\ve=1$.

For an ordering $P\in X_F$ we denote by $F_P$ a real closure of $F$ at $P$. By an 
\emph{$F$-algebra with involution} we mean a pair $(A,\s)$ where $A$ is a 
finite-dimensional simple $F$-algebra with centre a field $Z(A)$, equipped with an involution  $\s:A\to A$, such that $F = Z(A) \cap \Sym(A,\s)$.  Observe that $\dim_F Z(A) =: \kappa 
\le 2$.
If $A$ is a division algebra, we call $(A,\s)$ an \df{$F$-division algebra with involution}. Note that after scalar extension 
from $F$ to
$F_P$, $A$ may decompose into the direct product of two central simple $F_P$-algebras. By \cite[Proposition~2.14]{BOI} we may assume that $\s\ox F_P$ is the exchange involution in this case.

When $\kappa=1$, we say that 
$\s$ is \df{of the first kind}. When $\kappa=2$, we say that $\s$ is \df{of the second kind} (or of \df{unitary type}).  Note that $\s|_{Z(A)}$ is the non-trivial $F$-automorphism of $Z(A)$ in this case.
Assume $\kappa=1$. Then $\dim_FA=m^2$ for some $m\in \N$ and $\s$ is either of \df{orthogonal type}
(if $\dim_F \Sym(A,\s)= m(m+1)/2$)
 or  of \df{symplectic type} (if $\dim_F \Sym(A,\s)= m(m-1)/2$).  

It follows from the structure theory of $F$-algebras with involution  that $A$ is isomorphic to a full matrix algebra $M_n(D)$ for a unique $n\in \N$ and an $F$-division algebra $D$
(unique up to isomorphism) which is equipped with an involution $\vt$ of the same kind as $\s$, cf. \cite[Theorem~3.1]{BOI}. 

Let $(A,\s)$ and $(B,\tau)$ be $F$-algebras with involution of the same kind. If $A$ and $B$ are Brauer equivalent, then $(A,\s)$ and $(B,\tau)$ are Morita equivalent, cf. \cite[Example~1.4]{Dejaiffe-98}.

All forms in this paper are assumed to be non-singular and are identified with their Witt equivalence classes.
We write $+$ for the sum in both $W(F)$ and $W_\ve(A,\s)$. The group 
$W_\ve(A,\s)$ is a $W(F)$-module and we denote the product of $q\in W(F)$ and $h\in W_\ve(A,\s)$ by $q\cdot h$. 

Assume that $(A,\s)$ and $(B,\tau)$ are Morita equivalent $F$-algebras with involution of the same kind. It follows from \cite[Theorem~9.3.5]{Knus} (for full details, see \cite[Chapitre~2]{GB}) that
there exists an isomorphism $W_\ve(A,\s) \cong W_{\ve\ve_0} (B,\tau)$, where $\ve_0=1$ if $\s$ and $\tau$ are both orthogonal or both symplectic, and $\ve_0=-1$ otherwise. If $\s$ and $\tau$ are both unitary, then the isomorphism holds for any
$\ve_0\in \{-1, 1\}$, cf. \cite[Lemma~2.1$(iii)$]{A-U-Kneb}.

In the context of signatures  later on (Section~\ref{sec-signatures}), we will consider non-trivial morphisms from $W(A\ox_F F_P, \s\ox\id)$ to $\Z$ and therefore need  to know when  $W(A\ox_F F_P, \s\ox\id)$ is torsion, which motivates the following definition.

\begin{defi}\label{nil} 
Let $(A,\s)$ be an $F$-algebra with involution. We define the set of 
\emph{nil-orderings} of $(A,\s)$ as follows
\[\Nil[A,\s]:=\{ P\in X_F \mid W(A\ox_F F_P, \s\ox \id) \text{ is torsion}\}.\]
\end{defi}

For convenience we also introduce 
\[\wt X_F:=X_F \sm \Nil[A,\s],\]
which does not indicate the dependency on $(A,\s)$ in order to avoid cumbersome notation. Recall from 
\cite[Corollary~6.5]{A-U-Kneb} that $\Nil[A,\s]$ and thus also $\wt X_F$ are clopen in $X_F$.

For $a,b\in F^\x$, we denote by $(a,b)_F$ the quaternion $F$-algebra with generators $i$ and $j$ such that $i^2=a$ and $j^2=b$.

Let  $R$ be a real closed field and let $(A,\s)$ be an $R$-algebra with involution. It follows from well-known theorems of Wedderburn and Frobenius  and an application of Morita theory that $W(A,\s)$ 
is isomorphic to one of the following Witt groups:
\begin{equation}\label{list}
\begin{aligned}
W(R,\id) \cong W_{\pm 1}\bigl(R(\sqrt{-1}), -\bigr)   \cong W\bigl((-1,-1)_R, -\bigr) &\cong \Z \\
W_{-1}(R,\id) \cong W_{\pm 1}(R\x R, \swap)  \cong W_{\pm 1}((-1,-1)_R\x (-1,-1)_R, \swap) &\cong 0\\
W_{-1}\bigl((-1,-1)_R, -\bigr)&\cong \Z/2\Z
\end{aligned}
\end{equation}
where  $-$ denotes conjugation and $\swap$ denotes the exchange involution, cf. \cite[Lem\-ma~2.1 and \S3.1]{A-U-Kneb}.
(Note that we inadvertently omitted dealing with the case $((-1,-1)_R\x (-1,-1)_R, \swap)$ in \cite[\S 3.1]{A-U-Kneb}. Since
$W_{\pm 1}((-1,-1)_R\x (-1,-1)_R, \swap) \cong 0$, this omission does not affect our reasoning or our results in \cite{A-U-Kneb}.)

\begin{prop}\label{nil-prop} 
Let $(A,\s)$ be an $F$-algebra with involution.
\begin{enumerate}[$(1)$]
\item We have
\[\Nil[A,\s]:=\{ P\in X_F \mid W(A\ox_F F_P, \s\ox \id) \text{ is isomorphic to $0$ or $\Z/2\Z$}\}.\]

\item Let $(B,\tau)$ be an $F$-algebra with involution such that $A$ and $B$ are Brauer equivalent
and $\s$ and $\tau$ are of the same type. Then $\Nil[B,\tau]=\Nil[A,\s]$.

\item Let $L$ be an algebraic extension of $F$ and let $P \in X_F$. For every extension $Q$ of $P$ to $L$ we have  $Q\in \Nil[A\ox_F L,\s\ox\id]$ if and only if $P\in \Nil[A,\s]$.

\item Let $P\in X_F$. Then $P \in \Nil[A,\s]$ if and only if  any morphism from $W(A\ox_F F_P, \s\ox\id)$ to $\Z$ is identically zero. 

\end{enumerate}
\end{prop}

\begin{proof} 
(1) Let $P\in X_F$.
The statement follows from considering the list \eqref{list} with $R=F_P$.  

(2) Let $P\in X_F$. By the assumption and Morita theory, $W(A\ox_F F_P,\s\ox\id)\cong 
W(B\ox_F F_P, \tau\ox\id)$. 

Statement (3)  follows from the observation that $(A\ox_F L)\ox_L L_Q \cong A\ox_F F_P$ and 
(4)  follows from \cite[Theorem~4.1]{LU1}.
\end{proof}

We remark that by Proposition~\ref{nil-prop}  our exposition 
of nil-orderings in this paper is equivalent to those in \cite{A-U-Kneb} and \cite{A-U-prime}.

\begin{defi}\label{real} 
Let $(D,\vt)$ be an $F$-division algebra with involution and let $P\in X_F$. We say that $(D,\vt)$ is \emph{$(F,P)$-real} (or simply \emph{$F$-real} in case $F$ has a unique ordering) if
\[(D,\vt)\in \{(F,\id), (F(\sqrt{-d}), -),  \bigl((-a,-b)_F,-\bigr)\},\]
where $a,b,d \in P$ and $-$ denotes conjugation.
\end{defi}

\begin{lemma}\label{lem2.5} 
Let $(D,\vt)$ be an $F$-division algebra with involution such that $\deg D\leq 2$ and let $P\in X_F$. The following statements are equivalent:
\begin{enumerate}[$(1)$]
\item $(D,\vt)$ is $(F,P)$-real;
\item $(D\ox_F F_P, \vt\ox\id)$ is $F_P$-real;
\item $P\not\in \Nil[D,\vt]$.
\end{enumerate} 
\end{lemma}

\begin{proof} (1)$\Rightarrow$(2) is clear. (2)$\Rightarrow$(3): $W (D\ox_F F_P, \vt\ox\id)$ is not torsion 
by \eqref{list}
 and thus $P\not\in \Nil[D,\vt]$.
(3)$\Rightarrow$(1): since $\deg D\leq 2$, this follows from an examination of \eqref{list}, Definition~\ref{real} and Proposition~\ref{nil-prop}(1).  
\end{proof}

\section{Signatures Revisited}\label{sec-signatures}

We give a self-contained presentation of signatures of hermitian forms  in the following paragraphs, 
using results obtained in \cite{A-U-Kneb} and \cite{A-U-prime} (where we called them $H$-signatures).

Let $(A,\s)$ be an $F$-algebra with involution and let $P\in \wt X_F$. 
Using Proposition~\ref{nil-prop}(1), Lemma~\ref{lem2.5} and \eqref{list} we 
obtain the sequence of group morphisms
\begin{equation}\label{seq}
\begin{small}
\xymatrix{
W(A,\s) \ar[r]^--{r_P} & W(A\ox_F F_P, \s\ox\id) \ar[r]^--{\mu_P} & W(D_P,\vt_P)\ar[r]^--{\rho_P} 
& W(F_P)\ar[r]^--{\sign}  & \Z,
}
\end{small}
\end{equation}
where $r_P$ is the canonical extension of scalars map, $(D_P, \vt_P)$ is an $F_P$-real division algebra with involution
which is Morita equivalent   (via scaling and collapsing, see
  \cite[Section 2.4]{A-U-Kneb})  to $(A\ox_F F_P , \s\ox\id)$,
$\mu_P$ is an isomorphism induced by Morita theory, $\rho_P$ is defined by
$\rho_P(h)(x):= h(x,x)$ for all $h \in W(D_P,\vt_P)$ (cf. \cite{Jacobson-40}) and $\sign$ is the usual 
Sylvester signature of quadratic forms at the unique ordering of $F_P$.

Following \cite{BP2} we would like to use the map defined by \eqref{seq} from $W(A,\s)$ to $\Z$ as our definition 
of signature. However, a different choice of the map $\mu_P$ may result in a sign change (and can always be used 
to induce a sign change), so that both the map defined by \eqref{seq} and its opposite give possible signatures, 
cf. \cite[\S3.2]{A-U-Kneb}. It is tempting to fix the sign by choosing $\mu_P$ such that the hermitian form $\la 1\ra_\s$
is mapped to a positive integer, as in the quadratic forms case. This is not always possible though since there are
instances where $\la 1\ra_\s$ is mapped to zero, cf. 
\cite[Remark~3.11]{A-U-Kneb}.

In \cite[Proposition~3.2]{A-U-prime} we showed that 
there exists a hermitian form $\eta\in W(A,\s)$ such that for all $P\in \wt X_F$,
\begin{equation}\label{sign_form} 
\sign (\rho_P\circ\mu_P )(\eta\ox F_P) \not= 0.
\end{equation}
We call $\eta$ a \emph{reference form} for $(A,\s)$. (In \cite{A-U-Kneb} we used a tuple of forms, rather than one form.)

The purpose of the next definition is to use the form $\eta$ instead of the form $\la 1\ra_\s$ in order to select the signature
map that sends $\eta$ to a positive integer.

\begin{defi} Let $\eta\in W(A,\s)$ be a reference form for $(A,\s)$, let $h\in W(A,\s)$, let $P\in \wt X_F$, let $\ell=\dim_{F_P} D_P$ and let $\delta_P^{\eta}\in \{-1,1\}$ be the sign of  $\sign (\rho_P\circ\mu_P )(\eta\ox F_P)$. 
We define the \emph{signature of $h$ at $P\in X_F$}, $\sign_P^{\eta} h$, by
\[\sign_P^{\eta} h:=\frac{1}{\ell} \delta_P^{\eta} \sign (\rho_P\circ\mu_P )(h\ox F_P) \]
whenever $P\in \wt X_F$ and $\sign_P^{\eta} h :=0$ if $P\in\Nil[A,\s]$, cf. Proposition~\ref{nil-prop}(4).
\end{defi}

\begin{remark}\label{going-up}\mbox{} 
\begin{enumerate}[(1)]
\item We showed in \cite[\S3.3]{A-U-Kneb} that $\sign_P^{\eta}$ only depends on $P$ and $\eta$ and in \cite[\S7]{A-U-prime} that  signatures  correspond canonically to a natural class of morphisms from $W(A,\s)$ to $\Z$.
\item If $L$ is a finite extension of $F$, it follows from \eqref{sign_form} that $\eta\ox_F L$ is a reference form for $(A\ox_F L, \s\ox\id)$. Moreover, if $R$ is an ordering on $L$ that extends $P\in X_F$, then 
\[\sign_R^{\eta\ox L} (h\ox L) = \sign_P^{\eta} h\]
for all $h\in W(A,\s)$.
\item Let $P\in \wt X_F$. If $\eta'$ is another reference form for $(A,\s)$, then easy computations show that
\[\sign_P^{\eta} h=\delta_P^{\eta} \delta_P^{\eta'} \sign_P^{\eta'} h \]
for all $h\in W(A,\s)$ and
\[\delta_P^{\eta} \delta_P^{\eta'} = \sgn(\sign_P^{\eta'}\eta ) = \sgn(\sign_P^{\eta} \eta'),\]
where $\sgn$ denotes the sign map.

\item For $P \in \wt X_F$ we have $A \ox_F F_P \cong M_\ell(D_P)$ for some $\ell
  \in \N$  and  $(A \ox_F F_P,
  \s \ox \id)$ is Morita equivalent to $(D_P, \vt_P)$. This is
  already  the case  in a finite extension of $F$ in $F_P$, so there is
  a finite ordered extension $(L,Q)$ of $(F,P)$ such that $(A \ox_F L, \s \ox
  \id)$ is Morita equivalent to an $(L,Q)$-real division algebra with
  involution.
\item   Recall from \cite[Proposition~3.6]{A-U-Kneb} that $\sign_P^\eta(q\cdot h) =\sign_P q \cdot \sign_P^\eta h$ for
$h\in W(A,\s)$ and $q\in W(F)$. 
\end{enumerate}
\end{remark}

\begin{lemma}\label{Jacobson} 
Let $P\in X_F$, let $(D,\vt)$ be an $(F,P)$-real division algebra with involution   and let $\ell= \dim_F D$.  
Then $\eta=\la 1\ra_\vt$ is a
reference form for $(D,\vt)$, and considering
the group morphism
$\rho: W(D,\vt)\to W(F)$,
where $\rho(h)(x):=h(x,x)$, we have 
\[\sign_P^\eta h=\frac{1}{\ell} \sign_P \rho(h).\]
\end{lemma}

\begin{proof} 
The fact that $\eta$ is a reference form for $(D,\vt)$ follows from 
\[X_F\setminus\Nil[D,\vt] = \begin{cases}
X_F & \text{if } (D,\vt)=(F,\id)\\
\{Q\in X_F \mid d \in Q\} & \text{if } (D,\vt)=(F(\sqrt{-d}), -)\\
\{Q\in X_F \mid a,b \in Q\} & \text{if } (D,\vt)=((-a,-b)_F, -)\\
\end{cases}.
\]
Note that $P\not\in\Nil[D,\vt]$ by 
Lemma~\ref{lem2.5} and
consider the diagram

\begin{equation}
\begin{split}
\xymatrix{
W(D,\vt) \ar[d]_--{\rho} \ar[r]^--{r_P} & W(D\ox_F F_P, \vt\ox\id) \ar[r]^--{\id} & W(D\ox_F F_P, \vt\ox\id)
\ar[d]^--{\rho_P} \\
W(F)\ar[rr]^--{r'_P} & & W(F_P)\ar[r]^--{\sign}  & \Z
}
\end{split}
\end{equation}
where we used the notation from the sequence \eqref{seq} and $r_P'$ denotes extension of scalars. The square on the left commutes:
Let $h\in W(D,\vt)$. 
Then $\rho(h)\ox F_P = \rho_P(h\ox F_P)$ since 
\[(\rho(h)\ox F_P) (x\ox1)=\rho(h)(x)\ox 1 =h(x,x)\ox 1= (h\ox F_P) (x\ox 1, x\ox 1)=\rho_P(h\ox F_P)(x\ox 1).\]
It follows that 
\begin{align*}
\sign^\eta_P h 
&= \frac{1}{\ell} \ve_P \sign (\rho_P\circ\id)(h\ox F_P)\\
&= \frac{1}{\ell} \sign \rho(h)\ox F_P\\  
&= \frac{1}{\ell} \sign_P \rho(h),
\end{align*}
where $\ve_P  = \sgn\Bigl(\sign (\rho_P\circ\id )(\eta\ox F_P)\Bigr)=  \sgn\bigl(\sign (\ell\x \la 1\ra)\bigr)=   1$.
\end{proof}

\begin{lemma}\label{sign-surj}
 Let $P\in X_F$ and let $(D,\vt)$ be an $(F,P)$-real division algebra with involution. Let $\eta$ be a reference form for $(D,\vt)$.  The signature map
\[ \sign_P^\eta :W(D,\vt) \too \Z\]
is surjective.
\end{lemma}

\begin{proof} By hypothesis $P\not\in\Nil[D,\vt]$.
Let $h=\la 1\ra_\vt$. Then $\sign_P^\eta h = \pm 1$ by Remark~\ref{going-up}(3) and Lemma~\ref{Jacobson}.
\end{proof}

\section{Stability Index of Algebras with Involution}

In this section, we fix an $F$-algebra with involution $(A,\s)$ and a reference form $\eta$ for $(A,\s)$. 
 For every $h \in W(A,\s)$ we denote by $\sign^\eta h$ the total signature map from $X_F$ to $\Z$ and remark that it is  continuous, cf. \cite[Theorem~7.2]{A-U-Kneb}.
If $h \in W(A,\s)$, we
have $\sign^\eta h = 0$ on $\Nil[A,\s]$. 
Therefore it is convenient to introduce
the  notation
\[\wt C(X_F, \Z) := \{f \in C(X_F, \Z) \mid f=0 \text{ on } \Nil[A,\s]\}.\]
Note that $\wt C(X_F, \Z)$ depends on the Brauer class of $A$ and the type of $\s$, but indicating this would make the 
notation cumbersome. Since $(A,\s)$ is fixed, no confusion should arise.  
For $P\in X_F$ and  a field extension $L$ of $F$, we define
\[X_L/P:=\{Q \in X_L \mid P \subseteq Q\}.\]

\begin{lemma}\label{power-of-2-local}
   Let $P \in \widetilde
  X_F$.  Then there exists a hermitian form $h_P \in W(A,\sigma)$ and a nonnegative
  integer $\ell_P$ such that $\sign^\eta_P h_P = 2^{\ell_P}$.
\end{lemma}

\begin{proof}
  By Remark \ref{going-up} there exists a finite ordered field extension $(L,Q)$ of
  $(F,P)$ such that  $(A\ox_F L, \s\ox\id)$ is Morita equivalent to an
  $(L,Q)$-real division algebra with involution.
  Let $Q_0 \in
  X_L/P$. By Lemma~\ref{sign-surj} and \cite[Theorem~4.2]{A-U-prime}
  there exists a hermitian form $h$ over $(A
  \otimes_F L, \s \otimes \id)$ such that $\sign^{\eta\ox L}_{Q_0} h = 1$. 
  Since $X_L/P$ is finite, there exist $a_1, \ldots, a_r \in L$ such that $\{Q_0\}
  = \{Q \in X_L/P \mid a_1, \ldots, a_r \in Q\}$. Using the notation from \cite[\S 5.3]{A-U-Kneb},
  let $h_P := \trs{A\ox_F L}\bigl( \la\!\la a_1,
  \ldots, a_r \ra\!\ra \cdot h\bigr)$. It follows from the (hermitian) Knebusch trace formula 
  \cite[Theorem~8.1]{A-U-Kneb} 
  that
  \begin{align*}
    \sign^\eta_P h_P &= \sum_{Q \in X_L/P} \sign^{\eta \otimes L}_Q \bigl( \la\!\la a_1,
      \ldots, a_r \ra\!\ra \cdot h \bigr) \\
      &= 2^r \sign^{\eta \otimes L}_{Q_0} h\\
      &= 2^r.\qedhere
  \end{align*}
  \end{proof}

\begin{lemma}\label{h-naught}
  There exists a
  hermitian form $h_0 \in W(A,\sigma)$ and a nonnegative integer $k_0$ such that
  $\sign^\eta_P h_0 = 2^{k_0}$ for every $P \in \widetilde X_F$.
\end{lemma}

\begin{proof}
  By Lemma \ref{power-of-2-local}, for every $P \in \widetilde X_F$ there exists $\ell_P \in \N\cup\{0\}$,
  $U_P$ clopen in $\wt X_F$ containing $P$, and $h_P \in W(A,\s)$ such that $\sign^\eta h_P = 2^{\ell_P}$
  on $U_P$ (simply take $U_P = (\sign^\eta h_P)^{-1}(2^{\ell_P})$).

  Therefore $\widetilde X_F = \bigcup_{P \in \widetilde X_F} U_P = \bigcup_{i=1}^n
  U_{P_i}$ since $\widetilde X_F$ is compact. By removing  the intersections
  of the sets $U_{P_i}$ we obtain $\widetilde X_F = \bigdotcup_{i=1}^r C_i$ where each
  $C_i$ is clopen and $\sign^\eta \psi_i = 2^{\ell_i}$ on $C_i$ for  hermitian forms
  $\psi_i \in W(A,\s)$, $i=1,\ldots, r$.

  Let $q_i \in W(F)$ be such that $\sign q_i$ is equal to $2^{s_i}$ on $C_i$
  (for some nonnegative  integer $s_i$) and $0$ elsewhere, cf. \cite[VIII, Lemma~6.10]{Lam}. Then
  $\sign^\eta(q_i\cdot \psi_i)$ is equal to $2^{s_i+\ell_i}$ on $C_i$ and $0$ elsewhere. Taking
  $k_0 = \max\{s_1+\ell_1, \ldots, s_r+\ell_r\}$ and multiplying $q_i\cdot \psi_i$ by a suitable power of $2$, we obtain a form $h_i \in
  W(A,\s)$ such that $\sign^\eta h_i = 2^{k_0}$ on $C_i$ and $0$ elsewhere. It
  follows that for $h_0 := h_1 + \cdots + h_r$, $\sign^\eta h_0 = 2^{k_0}$ on
  $\widetilde X_F$.
\end{proof}

\begin{prop}\label{cokernel}
  Let $k_0 \in \N\cup\{0\}$ and $h_0 \in W(A,\s)$ be such that $\sign^\eta h_0 = 2^{k_0}$ on
  $\wt X_F$.
  \begin{enumerate}[$(1)$]
    \item Let $q \in W(F)$. Then there exists $h \in W(A,\s)$ such that
      $\sign(2^{k_0}q) = \sign^\eta(h)$ on $\wt X_F$.
    \item Let $f \in \wt C( X_F, \Z)$. Then there exists $n \in \N$ such that $2^n f \in
      \im(\sign^\eta)$.
  \end{enumerate}
\end{prop}

\begin{proof}
(1) We take $h = q\cdot h_0$.

 (2) We know that there exists $m \in \N\cup\{0\}$ such that $2^mf = \sign(q)$ for some
      $q \in W(F)$, cf. \cite[VIII, Lemma~6.9]{Lam}. Then $2^{m+k_0}f = \sign^\eta(q\cdot h_0)$.
\end{proof}

\begin{defi}\mbox{}
  \begin{enumerate}[(1)]
    \item We define  $\st^\eta(A,\s)$ to be the smallest nonnegative integer $k$ such that
      \[2^k\cdot  \wt C(X_F,\Z) \subseteq \im(\sign^\eta)\]
      if such an integer exists, and infinity otherwise.
    \item We define $S^\eta(A,\s)$  to be the cokernel of the total signature map
      \[\sign^\eta : W(A,\s) \rightarrow  \wt C(X_F,\Z).\]
  \end{enumerate}
\end{defi}

The following corollary follows immediately from Proposition~\ref{cokernel}.
\begin{cor}\label{coker}
  The group $S^\eta(A,\s)$ is $2$-primary torsion, and its exponent is $2^{\st^\eta(A,\s)}$ \textup{(}with the convention that $2^\infty = \infty$\textup{)}.
\end{cor}

\begin{prop}\label{stab-iso}
  Let $\eta'$ be another reference form for $(A,\s)$. Then $S^\eta(A,\s) \cong
  S^{\eta'}(A,\s)$. In particular $\st^\eta(A,\s) = \st^{\eta'}(A,\s)$.
\end{prop}

\begin{proof}
  By \cite[Proposition~3.3$(iii)$]{A-U-prime}, there exists $f \in C(X_F,\{-1,1\})$ such that $\sign^\eta =
  f\cdot \sign^{\eta'}$. Define
  \[\begin{array}{rcl}
      \xi :  \wt C(X_F,\Z) & \longrightarrow & \wt  C(X_F,\Z)/\im \sign^\eta \\
      g & \longmapsto & f\cdot g + \im \sign^\eta
    \end{array}\]
  The map $\xi$ is a surjective morphism of groups since $f$ is invertible in $C(X_F,\Z)$. Moreover, $g \in \ker \xi$ if and only
  if $f\cdot g \in \im \sign^\eta = f\cdot \im \sign^{\eta'}$, so $\ker \xi = \im
  \sign^{\eta'}$ and $\xi$ induces an isomorphism from $S^{\eta'}(A,\s)$ to
  $S^\eta(A,\s)$.
\end{proof}

\begin{defi}  
We denote by $S(A,\s)$ a representative of the isomorphism class of  $S^\eta(A,\s)$ and 
call it the \emph{stability group} of $(A,\s)$. We let $\st(A,\s):= \st^\eta(A,\s)$ and
call it  the \emph{stability index} of $(A,\s)$. 
\end{defi}

\begin{prop}
  Let $h_0 \in W(A,\s)$ and $k_0 \in \N\cup\{0\}$ be as in Lemma~\ref{h-naught}.
Then 
\[\st(A,\s) \le \st(F)   + k_0,\]
where $\st(F)$ denotes the stability index of $F$.
  \end{prop}

\begin{proof} 
Assume that $\st(F)$ is finite.
  Let $f \in  \wt C(X_F, \Z)$. Then there exists $q \in W(F)$ such that $2^{\st(F)} f = \sign
  q$, and thus $\sign^\eta(q\cdot h_0) = 2^{\st(F) + k_0}f$.
\end{proof}

\begin{prop}
  Let $(A,\s)$ and $(B,\tau)$ be two Morita equivalent $F$-algebras with involution.  Then $S(A,\s) \cong S(B,\tau)$ and $\st(A,\s) =
  \st(B,\tau)$.
\end{prop}
\begin{proof}
  It suffices to prove the first part of the statement, but this follows
  immediately from \cite[Theorem~ 4.2]{A-U-prime}.
\end{proof}

The following theorem extends a classical result in quadratic form theory (cf. \cite[(1.6)]{Broecker-74}) to algebras with involution:

\begin{thm}\label{thm:main} 
Let $(A,\s)$ be an $F$-algebra with involution and let $W_\tor(A,\s)$ denote the torsion subgroup of $W(A,\s)$.
The sequence
\begin{equation*}
\xymatrix{
0 \ar[r] &   W_\tor(A, \sigma ) \ar[r] & W(A, \sigma) \ar[r]^--{\sign^\eta} & \wt C(X_F,\mathbb{Z}) \ar[r] & S(A, \sigma )\ar[r] & 0
}
\end{equation*}
is exact. The groups $W_\tor(A, \sigma )$ and $S(A, \sigma )$ are $2$-primary torsion groups.
\end{thm}

\begin{proof} This follows from \cite[Theorem~4.1]{LU1}, \cite[Theorem~5.1]{Sch1} 
 and Corollary~\ref{coker}. 
\end{proof}

\begin{examples} In each of the examples below, $A$ is a quaternion division algebra.
\begin{enumerate}[(1)]
\item Let $F=\R$, $A=(-1,-1)_F$ and $\s$ quaternion conjugation. Then $\wt X_F=X_F$ and
 $\im\sign^\eta\cong \Z\cong \wt C(X_F,\Z)$. Hence $S(A,\s)\cong \{0\}$ and $\st(A,\s)=0$. Note that $\st(\R)=0$.

\item Let $F=\R$, $A=(-1,-1)_F$ and $\s$ orthogonal. Then $\wt  X_F = \varnothing$ and
 $\im\sign^\eta\cong \{0\} \cong \wt C(X_F,\Z)$. Hence $S(A,\s)\cong \{0\}$ and $\st(A,\s)=0$. 

\item Let $F=\Q(\sqrt{2})$, $A=(-1, -\sqrt{2})_F$ and $\s$ quaternion conjugation. Then $\wt X_F$ consist of the unique ordering that contains $\sqrt{2}$ and 
 $\im\sign^\eta\cong \Z\x \{0\} \cong \wt C(X_F,\Z)$. Hence $S(A,\s)\cong \{0\}$ and $\st(A,\s)=0$. Note that $\st(\Q(\sqrt{2}))=1$.

\item Let $F=\R(\!(x)\!)$, $A=(x, -1)_F$ and $\s$ orthogonal.  Then $\wt X_F$ consist of the unique ordering that contains $x$ and  $\im\sign^\eta\cong 2\Z\x \{0\}$ and $\wt C(X_F,\Z)\cong \Z \x \{0\}$. Hence $S(A,\s)\cong \Z/2\Z$ and $\st(A,\s)=1$. Note that $\st(\R(\!(x)\!))=1$.
\end{enumerate}
\end{examples}

We now consider the total signature  $\sign^\eta h$ of a hermitian form $h\in W(A,\s)$. Since this is a continuous function, there exists a nonnegative integer $k$ such that $2^k\sign^\eta (h)$ is the total signature of some quadratic form over $F$. In the next two results we will show that $k$ can be chosen independently of $h$.

\begin{lemma}\label{rel-part}
  There exist disjoint clopen subsets $U_1, \ldots, U_t$ of $\widetilde X_F$ and
  positive integers $n_1, \ldots, n_t $ such that $\widetilde X_F = U_1 \cup
  \cdots \cup U_t$ and for every $h \in W(A,\s)$ and $i \in \{1,\ldots, t\}$,
  there exists $q_i \in W(F)$ such that $\sign(q_i) = 2^{n_i}\sign^\eta(h)$ on $U_i$.
\end{lemma}

\begin{proof}
  By Remark \ref{going-up}, for every $P \in \widetilde X_F$ there exists a
   finite ordered field extension  $(L_P, R_P)$ of $(F,P)$ such that $(A
   \otimes_F L_P, \s\ox\id)$ is Morita equivalent to an $(L_P,R_P)$-real
   division algebra with involution $(D_P, \vt_P)$. 
  
  Let $S_P = \{Q \in \widetilde X_F \mid Q \textrm{ extends to } L_P\}$. By
  \cite[Chapter~3, Theorem~4.4]{Sch} we have $S_P= \{Q \in \widetilde X_F \mid
  \sign_Q (\trs{L_P/F}\la 1 \ra) > 0\}$,  
  where $\trs{L_P/F} : W(L_P) \to W(F)$ denotes the Scharlau transfer.
  The set $S_P$ is therefore a clopen
  subset of $\widetilde X_F$ containing $P$, and by compactness there are $P_1,
  \ldots, P_t \in \widetilde X_F$ such that $\widetilde X_F = S_{P_1} \cup
  \cdots \cup S_{P_t}$. It follows that there are disjoint clopen sets $S_i
  \subseteq S_{P_i}$ (for $i=1, \ldots, t$) such that $\widetilde X_F = S_1 \dotcup
  \cdots \dotcup S_t$.

Let $i \in \{1,\ldots, t\}$ and write $S:=S_i$, $L:=L_{P_i}$. It suffices to prove the lemma for $S$ instead of $\wt X_F$.
  Let $Q\in S$ and let $R \in X_{L}/Q$. We have
  morphisms of Witt groups \[W(A\ox_F L, \s\ox\id) \stackrel{\mu}{\too}
  W(D,\vt) \stackrel{\rho}{\too} W(L),\] where $(D,\vt)$ is $(L,
  R)$-real  (note that  $R\not\in\Nil[D,\vt]$ since $Q\not\in\Nil[A,\s]$ by
  Proposition~\ref{nil-prop}), $\mu$ is an isomorphism induced by Morita
  equivalence and $\rho$ is the map of Lemma~\ref{Jacobson}.  Let $\eta_0=\la
  1\ra_{\vt} \in W(D, \vt)$ and let $\delta_R \in\{-1,1\}$ be the sign of $\sign_R^{\mu(\eta\ox
  L)} \eta_0$.  Let $h\in W(A,\s)$ be arbitrary. Then, using
  Remark~\ref{going-up}, we have
  \begin{align*}
    4\sign_Q^\eta h &= 4\sign_R^{\eta\ox L} (h\ox L) \\
    &= 4\sign_R^{\mu(\eta\ox L)} \mu(h\ox L) & \textrm{[by  \cite[Theorem~4.2]{A-U-prime}]}\\
    &= 4\sgn(\sign_R^{\mu(\eta\ox L)} \eta_0) \sign_R^{\eta_0} \mu(h\ox L) & \textrm{[by Remark~\ref{going-up}(3)]}   \\
    &= 4\delta_R \sign_R^{\eta_0} \mu(h\ox L)\\
    &= \delta_R\frac{4}{\ell} \sign_R \rho(\mu(h\ox L))    & \textrm{[by Lemma~\ref{Jacobson}]}\\
    &= \delta_R \sign_R\vf,    
  \end{align*}
  where $\ell=\dim_{L}D \in \{1,2,4\}$
  and $\vf:= \frac{4}{\ell}\rho(\mu(h\ox L)) \in W(L)$.
   
  Observe that for every $R \in X_{L}/Q$, $\sign_R^{\mu(\eta\ox L)} \eta_0 \not
  = 0$ since $\eta_0$ is a reference form for $(D,\vt)$.
  Since $X_{L}/Q$ is finite, there is a finite tuple $\bar a_Q$ of elements of  $L$ of length $\ell_Q$
  such
  that $H(\bar a_Q) \cap (X_{L}/Q)$  contains only one ordering $R_{Q}$
  (where $H(\bar a_Q)$ denotes the Harrison set defined by $\bar a_Q$). In
  particular
  \[\sum_{R \in X_{L}/Q} \sign_R^{\mu(\eta\ox L)} (\la\!\la \bar a_Q\ra\!\ra
    \cdot \eta_0) = 2^{\ell_Q} \sign_{R_{Q}}^{\mu(\eta\ox L)} \eta_0 \not =
  0.\]
  Define $\vf_{Q} := \la\!\la \bar a_Q \ra\!\ra \cdot \vf$. Then, by the (quadratic) Knebusch trace formula
   \cite[Chapter~3, Theorem~4.5]{Sch},
  \[\sign_Q (\trs{L/F} \vf_{Q}) = \sum_{R \in X_{L}/Q} \sign_R \vf_{Q} =
  2^{\ell_Q} \sign_{R_{Q}} \vf = 2^{\ell_Q + 2}
  \delta_{R_{Q}} \sign^\eta_Q h.\]
  Therefore the clopen subset of $X_F$
  \[U_Q := (\sign (\trs{L/F} \vf_{Q}) - 2^{\ell_Q + 2}
  \delta_{R_{Q}} \sign^\eta h)^{-1}(0)\]
  contains $Q$, and thus $S = \bigcup_{Q \in X_F} U_Q$. Since $S$ is compact
  we obtain $S = U_{Q_1} \cup \cdots \cup U_{Q_r}$ for some $Q_1, \ldots, Q_r
  \in S$, and for every $Q \in U_{Q_j}$:
  \[2^{\ell_{Q_j} + 2} \delta_{R_{Q_j}} \sign^\eta_Q h = \sign_Q
  (\trs{L/F} \vf_{Q_j}),\]
  and thus
  \[2^{\ell_{Q_j} + 2} \sign^\eta_Q h = \sign_Q
  (\delta_{R_{Q_j}} \trs{L/F}  \vf_{Q_j}).\]
  Since  $\delta_{R_{Q_j}} \trs{L/F} \vf_{Q_j} \in W(F)$, the result
  follows by taking clopen sets $U_j \subseteq U_{Q_j}$ such that $S = U_1
  \dotcup \cdots \dotcup U_r$.
  \end{proof}

\begin{thm}\label{rel-st}
  There exists $n_0 \in \N\cup\{0\}$ such that for every $h \in W(A,\s)$ there exists $q_h \in
  W(F)$ with $2^{n_0} \sign^\eta h = \sign q_h$.
  \end{thm}

\begin{proof}
  We use the terminology of Lemma~\ref{rel-part}. Let $r := n_1
  + \cdots + n_t$. Then for every $h \in W(A,\s)$ and every $i \in \{1, \ldots,
  t\}$ there exists $q_i \in W(F)$ such that $\sign q_i = 2^r \sign^\eta h$ on $U_i$.

  Also, there exist $m \in \N\cup\{0\}$ and,  for every $i\in \{1,\ldots,t\}$,  some $p_i \in W(F)$ such that $\sign
  p_i = 2^m$ on $U_i$ and $\sign p_i = 0$ on $X_F \setminus U_i$. Therefore 
  $ \sign (p_i q_i)= 2^m \sign q_i=2^{m+r} \sign^\eta h$ on $U_i$ and $\sign (p_i q_i) =
  0$ on $X_F \setminus U_i$, and we obtain for $i=1, \ldots, t$,
  \begin{align*}
    \sign (p_1 q_1 + \cdots + p_tq_t) &= \sign p_iq_i \text{ on } U_i \\
         &= 2^{m + r} \sign^\eta h \text{ on } U_i
  \end{align*}
  and thus
 \[     \sign (p_1 q_1 + \cdots + p_tq_t) = 2^{m + r} \sign^\eta h \text{ on } X_F\]
  (note that by construction the quadratic form $p_1 q_1 + \cdots + p_tq_t$ has
  zero signature on $\Nil[A,\s]$).
\end{proof}

\begin{defi} We define
  \[W_\rd(A,\s) := W(A,\s)/W_\tor(A,\s) \cong \sign^\eta(W(A,\s))\subseteq \wt C(X_F,\Z)\]
  and call it the \emph{reduced Witt group} of $(A,\s)$.
\end{defi}

In order to compare reduced hermitian forms and reduced quadratic forms, we also introduce
\[ \wt W_\rd(F):= \{q \in W_\rd(F) \mid \sign q=0 \text{ on } \Nil[A,\s]\}.\]
Observe that $W_\rd(A,\s)$ is a $W_\rd(F)$-module and also a $\wt W_\rd(F)$-module in the natural way.

\begin{prop}
  Let $h_0 \in W(A,\s)$ and $k_0 \in \N\cup\{0\}$ be such that $\sign^\eta(h_0) = 2^{k_0}$ on $\wt X_F$ (cf. Lemma~\ref{h-naught}).   With notation as in Theorem~\ref{rel-st},
  the maps
  \[\wt W_\rd(F) \longrightarrow W_\rd(A,\s), \ q \longmapsto qh_0\]
  and
  \[W_\rd(A,\s) \longrightarrow \wt W_\rd(F), \ h \longmapsto q_h\]
are well-defined injective
  morphisms of $W_\rd(F)$-modules.
\end{prop}

\begin{proof}
  Identifying $W_\rd(F)$ and $W_\rd(A,\s)$ with the images of
  $\sign$ and $\sign^\eta$, we see that the first map is simply  multiplication
  by $2^{k_0}$.

  The second map is
  well-defined because $h_1 = h_2$ in $W_\rd(A,\s)$ is equivalent to $\sign^\eta h_1 =
  \sign^\eta h_2$, which implies $\sign q_{h_1} = \sign q_{h_2}$, so $q_{h_1} =
  q_{h_2}$ in $\wt W_\rd(F)$. It is also easy to check that it is an
  injective morphism of $W_\rd(F)$-modules.
\end{proof}

We finish this paper by pointing out some difficulties that need to be overcome in order to further the study of the stability index of algebras with involution.

In the quadratic form literature one can find important links between the stability index of the field $F$ and the powers of the fundamental ideal $I^n(F)$, which crucially depend on Pfister forms (see for example \cite[Satz~3.17]{Broecker-74}).
Although it is possible to define $I(A,\s)$ as the submodule of forms of even rank in $W(A,\s)$ (cf. \cite[\S 2.2]{BP1} or 
\cite[\S 3.2.1]{GB}), the lack of a tensor product of hermitian forms in general   is a  
serious obstacle to the development of  analogous concepts and connections for Witt groups of algebras with involution. 

Another issue is the following: the quadratic Pfister form $\la 1, a\ra$ has signature $2$ on $H(a)$ and signature $0$ on $H(-a)$, which is a fundamental observation when considering   $\st(F)$. In contrast, this behaviour cannot in general be replicated with hermitian forms since 
the signature of the form $\la 1\ra_\s$ may not be constant  and in addition may take values which are not in $\{-1,1\}$.

\def\cprime{$'$}
\providecommand{\bysame}{\leavevmode\hbox to3em{\hrulefill}\thinspace}
\providecommand{\MR}{\relax\ifhmode\unskip\space\fi MR }
\providecommand{\MRhref}[2]{%
  \href{http://www.ams.org/mathscinet-getitem?mr=#1}{#2}
}
\providecommand{\href}[2]{#2}


\begin{thebibliography}{10}

\bibitem{A-U-Kneb}
V. Astier and T. Unger, \emph{Signatures of hermitian forms and the
  {K}nebusch trace formula}, Math. Ann. \textbf{358} (2014), no.~3-4, 925--947.


\bibitem{A-U-prime}
\bysame, \emph{Signatures of hermitian forms and ``prime ideals'' of {W}itt
  groups}, Adv. Math. \textbf{285} (2015), 497--514.

\bibitem{BP1}
E.~Bayer-Fluckiger and R.~Parimala, \emph{Galois cohomology of the classical
  groups over fields of cohomological dimension {$\leq 2$}}, Invent. Math.
  \textbf{122} (1995), no.~2, 195--229. 

\bibitem{BP2}
\bysame, \emph{Classical groups and the {H}asse principle}, Ann. of Math. (2)
  \textbf{147} (1998), no.~3, 651--693. 

\bibitem{Broecker-74}
L.  Br{\"o}cker, \emph{Zur {T}heorie der quadratischen {F}ormen \"uber
  formal reellen {K}\"orpern}, Math. Ann. \textbf{210} (1974), 233--256.


\bibitem{Dejaiffe-98}
I.~Dejaiffe, \emph{Somme orthogonale d'alg\`ebres \`a involution et alg\`ebre
  de {C}lifford}, Comm. Algebra \textbf{26} (1998), no.~5, 1589--1612.


\bibitem{GB}
N. Grenier-Boley, \emph{Groupe de witt d'une alg{\`e}bre simple centrale
  {\`a} involution}, PhD thesis, Universit{\'e} de Franche-Comt{\'e}, 2004.

\bibitem{Jacobson-40}
N.~Jacobson, \emph{A note on hermitian forms}, Bull. Amer. Math. Soc.
  \textbf{46} (1940), 264--268. 

\bibitem{Knus}
M.-A. Knus, \emph{Quadratic and {H}ermitian forms over rings}, Grundlehren
  der Mathematischen Wissenschaften, vol. 294, Springer-Verlag, Berlin, 1991. 

\bibitem{BOI}
M.-A. Knus, A.  Merkurjev, M. Rost, and J.-P. Tignol,
  \emph{The book of involutions}, American Mathematical Society Colloquium
  Publications, vol.~44, American Mathematical Society, Providence, RI, 1998.


\bibitem{Lam}
T.~Y. Lam, \emph{Introduction to quadratic forms over fields}, Graduate Studies
  in Mathematics, vol.~67, American Mathematical Society, Providence, RI, 2005.

\bibitem{LU1}
D. W. Lewis and T.  Unger, \emph{A local-global principle for algebras
  with involution and {H}ermitian forms}, Math. Z. \textbf{244} (2003), no.~3,
  469--477. 

\bibitem{Sch1}
W. Scharlau, \emph{Induction theorems and the structure of the {W}itt
  group}, Invent. Math. \textbf{11} (1970), 37--44. 

\bibitem{Sch}
\bysame, \emph{Quadratic and {H}ermitian forms}, Grundlehren der Mathematischen
  Wissenschaften, vol. 270,
  Springer-Verlag, Berlin, 1985. 

\end{thebibliography}
\end{document}